%% file: ricks-one-dim.tex
\documentclass{amsart}

\input{artcfg}
\renewcommand{\setp}[2]{\left\{#1 : #2\right\}}
\newcommand{\bdT}{\partial_T}
\newcommand{\bdc}{\partial_{\infty}}

\DeclareMathOperator{\dW}{d_{\W}}
\DeclareMathOperator{\CAT}{CAT}

\title[A Rank Rigidity Result for One-Dimensional Tits Boundaries]{A Rank Rigidity Result for CAT(0) Spaces with One-Dimensional Tits Boundaries}

\author{Russell Ricks}
\address{Binghamton University, Binghamton, New York, USA}
\email{ricks@math.binghamton.edu}

\thanks{The author would like to thank Ralf Spatzier for his support and many helpful discussions.
This material is based upon work supported by the National Science Foundation under Grant Number NSF 1045119.}

\begin{document}

\begin{abstract}
We prove the following rank rigidity result for proper CAT(0) spaces with one-dimensional Tits boundaries:
Let $\Gamma$ be a group acting properly discontinuously, cocompactly, and by isometries on such a space $X$.
If the Tits diameter of $\bd X$ equals $\pi$ and $\Gamma$ does not act minimally on $\bd X$, then $\bd X$ is a spherical building or a spherical join.
If $X$ is also geodesically complete, then $X$ is a Euclidean building, higher rank symmetric space, or a nontrivial product.
Much of the proof, which involves finding a Tits-closed convex building-like subset of $\bd X$, does not require the Tits diameter to be $\pi$, and we give an alternate condition that guarantees rigidity when this hypothesis is removed, which is that a certain invariant of the group action be even.
\end{abstract}

\maketitle

\section{Introduction}

The Rank Rigidity Theorem for nonpositively curved manifolds (see \cite{ballmann} or \cite{burns-spatzier}) states that if a group acts geometrically---that is, properly discontinuously, cocompactly, and by isometries---on a nonpositively curved Riemannian manifold, one of the following holds:
(a)  $M$ admits a rank one axis,
(b)  $M$ splits as nontrivial product, or
(c)  $M$ is a higher rank symmetric space.

$\CAT(0)$ spaces generalize nonpositive curvature from the Riemannian to the metric setting, allowing one to study nonpositive curvature without requiring a smooth manifold structure.
A well-known conjecture for $\CAT(0)$ spaces is the following generalization of the Rank Rigidity Theorem for compact nonpositively curved manifolds (cf. \cite{bb}).

\begin{conjecture*} [Rank Rigidity]
Let $X$ be a proper, geodesically complete $\CAT(0)$ space under a geometric action.
If $X$ does not admit a rank one axis, then $X$ is a higher rank
symmetric space or Euclidean building, or splits as a nontrivial product.
\end{conjecture*}

The Rank Rigidity Conjecture is known to hold among various classes of $\CAT(0)$ spaces.
In addition to the Riemannian case mentioned previously, Rank Rigidity holds for $\CAT(0)$ cube complexes \cite{caprace-sageev}.
It also holds for $\CAT(0)$ piecewise-smooth $2$-dimensional polyhedral complexes \cite{bbrin-2dim}.
A similar result holds for $\CAT(0)$ $3$-dimensional Euclidean polyhedral complexes \cite{bbrin-3dim}.

Each of these results requires an additional structure on the $\CAT(0)$ space---a Riemannian or polyhedral structure.
Our approach, in contrast, is to put only a dimension restriction on the Tits boundary of the $\CAT(0)$ space.
The Tits boundary consists of all the geodesic rays from a fixed basepoint, and it carries a natural metric called the Tits metric.
If $X$ is a proper $\CAT(0)$ space admitting a geometric action and the Tits boundary $\bdT X$ has dimension zero, then $X$ is $\delta$-hyperbolic and rank rigidity is trivial.
So studying the case where $\bdT X$ has dimension one is a natural first step toward the general conjecture.

Now, every $\CAT(0)$ piecewise-smooth $2$-dimensional polyhedral complex has $\dim(\bdT X) \le 1$, and every $\CAT(0)$ $3$-dimensional Euclidean polyhedral complex has $\dim(\bdT X) \le 2$, but in both cases the dimension may be $0$ or $1$.
On the other hand, the product of two $\delta$-hyperbolic proper $\CAT(0)$ spaces $X,Y$ of any dimension has $\dim \bdT (X \times Y) = 1$.
\corref{MC2} provides a version of rank rigidity that works directly from the dimension of the boundary and detects such products.

The set of geodesic rays with fixed basepoint, which forms the Tits boundary, comes with another natural structure, the cone topology.
This topology is coarser than the one induced by the Tits metric, and we will sometimes write $\bdc X$ to emphasize that we are using the cone topology.

When $X$ is a nonpositively curved Riemannian manifold (under a geometric group action),
the following three conditions are equivalent:
\begin{enumerate}
\item  \label{ro1}
$X$ admits a rank one axis
\item  \label{ro2}
$\diam (\bdT X) > \pi$
\item  \label{ro3}
the induced $\G$-action on $\bdc X$ is minimal---that is, the boundary contains no proper nonempty $\G$-invariant subsets that are closed in the cone topology%
\end{enumerate}
(see \cite{ballmann} or \cite{ricks-mixing} for a more detailed discussion).
On the other hand, for a $\CAT(0)$ space, while one may conjecture that all three conditions are equivalent, we only know that
\itemrefstar{ro1} holds if and only if both \itemrefstar{ro2} and \itemrefstar{ro3} hold.
In this paper, we prove rigidity assuming the stronger hypothesis that both \itemrefstar{ro2} and \itemrefstar{ro3} fail:

\begin{maintheorem}[\thmref{dimension 1, diameter pi, boundary statement}]	\label{MC1}
Let $\Gamma$ be a group acting geometrically on a proper $\CAT(0)$ space $X$
with $\dim(\bdT X) = 1$.
If $\diam (\bdT X) = \pi$ and $\Gamma$ does not act minimally on $\bdc X$, then $\bdT X$ is a spherical building or a spherical join. \end{maintheorem}

By a theorem of Leeb \cite{leeb}, we obtain rigidity for the original $\CAT(0)$ space.

\begin{maincorollary}[\corref{dimension 1, diameter pi, geodesically complete}]	\label{MC2}
Let $\Gamma$ be a group acting geometrically on a proper, geodesically complete $\CAT(0)$ space $X$ with $\dim(\bdT X) = 1$.
If $\diam (\bdT X) = \pi$ and $\Gamma$ acts minimally on $\bdc X$, then $X$ is a Euclidean building, higher rank symmetric space, or a nontrivial product. \end{maincorollary}

In fact, we prove the following result, stronger than \thmref{MC1}.

\begin{maintheorem}[\thmref{main theorem}]	\label{MT}
Let $\Gamma$ be a group acting geometrically on a proper $\CAT(0)$ space $X$.
Assume $\dim(\bdT X) = 1$ and $M \subsetneq \bdc X$ is a proper minimal subset of $\bdc X$.  Let $K \subset \bdT X$ be a folded round sphere, and let $\ell = \abs{M \cap K}$.  Then $1 \le \ell < \infty$.  Suppose that $M$ is chosen to minimize $\ell$, among all minimal sets $M \subset \bdc X$.  If $\ell = 1$, $\bdT X$ splits as a suspension.  If $\ell = 2$, $\bdT X$ splits as a spherical join.  If $\ell \ge 4$ is even, or if $\ell \ge 3$ is odd and $\diam (\bdT X) = \pi$, then $\bdT X$ either is a spherical building or splits as a spherical join. \end{maintheorem}

The proof of \thmref{MT} uses the Centers Lemma (\lemref{centers lemma}) from \cite{ricks-radius} along with a rigidity result (\thmref{lytchak rigidity}) of Lytchak \cite{lytchak05} about involutive sets in $\bdT X$.
The case $\ell$ is odd is the most complicated, and our proof involves constructing a large Tits-closed, convex subset of $\bdT X$, which has many properties of a spherical building.
A detailed description of this subset is given in \lemref{odd case}.

We arrive at the following comprehensive list of options for $\ell = \ell(\Gamma)$.

\begin{maintheorem}[\thmref{table}]	\label{MC3}
Let $\Gamma$ be a group acting geometrically on a proper $\CAT(0)$ space $X$ with $\dim(\bdT X) = 1$.
Let $K \subset \bdT X$ be a folded round sphere.
For each finite-index subgroup $\Gamma_0$ of $\Gamma$, let
$\ell(\Gamma_0) = \inf \abs{M \cap K}$,
where the infimum is taken over minimal nonempty, closed, $\Gamma_0$-invariant subsets $M$ of $\bdc X$.

The following table summarizes the complete situation (here $\min \ell(\Gamma_0)$ is taken over all finite-index subgroups $\Gamma_0$ of $\Gamma$):
\textnormal{
\begin{center}
\begin{tabular}{l l l}
\hline
$\bd X$ & possible $\ell(\Gamma)$ & $\min \ell(\Gamma_0)$ \\
\hline
circle & $1$, $2$, $3$, $4$, or $6$ & $1$ \\
suspension but not a circle & $1$ or $2$ & $1$ \\
spherical join but not a suspension & $2$ or $4$ & $2$ \\
irreducible spherical building & integers $\ge 3$ & $\ge 3$ \\
minimal & $\infty$ & $\infty$ \\
$\pi < \diam (\bdT X) \le \pi + \frac{\pi}{\ell(\Gamma)}$ & odd integers $\ge 3$ & odd $\ge 3$ \\
\hline
\end{tabular}
\end{center}} \end{maintheorem}

Note that no example of the last case ($\pi < \diam (\bdT X) \le \pi + \frac{\pi}{\ell(\Gamma)}$) is known, and (at least if $X$ is geodesically complete) would provide a counterexample to the Rank Rigidity Conjecture stated above.
Also note that the bound $\diam (\bdT X) \le \pi + \frac{\pi}{3}$ coincides with Guralnik and Swenson's bound \cite{gs} for $\dim (\bdT X) = 1$.

\section{The Spaces}

\label{spaces}

\subsection{CAT(1) and CAT(0) Spaces}

A metric space $X$ is called \defn{proper} if closed balls in $X$ are compact.

A \defn{geodesic} in a metric space $X$ is an isometric embedding $\R \to X$; a \defn{geodesic ray} is an isometric embedding $[0,\infty) \to X$; and a \defn{geodesic arc} (or \defn{geodesic segment}) is an isometric embedding $[a,b] \to X$ for some $a < b$ in $\R$.
A metric space $X$ is called \defn{geodesic} if every pair of distinct points $x,y \in X$ is connected by a geodesic in $X$; if every such geodesic is unique, $X$ is \defn{uniquely geodesic}.
A metric space is \defn{$\pi$-geodesic} if every pair of distinct points with distance $< \pi$ is connected by a geodesic.
When the geodesic arc between $x,y \in X$ is unique, we will denote it by $[x,y]$.

A metric space is \defn{geodesically complete} if every geodesic arc can be extended to a locally isometric embedding $\R \to X$.

A \defn{$\CAT(0)$} space is a uniquely geodesic metric space such that every geodesic triangle $\triangle (x,y,z)$ is thinner than the corresponding comparison triangle $\cl \triangle (x,y,z)$ in Euclidean $\R^2$.
More precisely, let $x,y,z \in X$ and find $\bar x, \bar y, \bar z \in \R^2$ such that $d(\bar x, \bar y) = d(x,y)$, $d(\bar x, \bar z) = d(x,z)$, and $d(\bar y, \bar z) = d(y,z)$; then $\cl \triangle (x,y,z) = \triangle (\bar x, \bar y, \bar z)$.
Then $\triangle (x,y,z)$ is thinner than $\cl \triangle (x,y,z)$ if for every $p \in [x,y]$ and $q \in [x,z]$, the corresponding points $\bar p \in [\bar x, \bar y]$ and $\bar q \in [\bar x, \bar z]$ with $d(\bar p, \bar x) = d(p,x)$ and $d(\bar q, \bar x) = d(q,x)$.
So $X$ is $\CAT(0)$ if for every $x,y,z \in X$, the triangle $\triangle (x,y,z)$ is thinner than $\cl \triangle (x,y,z)$.

A \defn{$\CAT(1)$} space is a $\pi$-geodesic metric space such that every geodesic triangle with perimeter $< 2\pi$ is thinner than the corresponding comparison triangle in $S^2$, the Euclidean $2$-sphere with the standard metric of constant curvature $1$.

The following facts about $\CAT(1)$ spaces are standard (\cite{bridson} is a nice reference).

\begin{lemma}			\label{unique geodesics}
Every pair of distinct points $x,y$ of distance $< \pi$ in a $\CAT(1)$ space $Y$ is joined by a unique geodesic arc $[x,y]$ in $Y$. \end{lemma}

\begin{lemma}			\label{local geodesics}
Every locally geodesic arc of length $\le \pi$ in a $\CAT(1)$ space is geodesic. \end{lemma}

A subset $C$ of a $\CAT(1)$ space is called \defn{convex} if every geodesic in $Y$ joining pairs of points in $C$ lies completely in $C$.

\begin{lemma}			\label{convex balls}
Every metric ball of radius $< \frac{\pi}{2}$ in a $\CAT(1)$ space is convex. \end{lemma}

\subsection{Angles and Boundaries}

Let $Y$ be a $\CAT(1)$ space $Y$ and $p \in Y$.  Given $x,y \in Y \setminus \set{p}$, the angle $\angle_p (x,y)$ is the limit of the corresponding angle in the comparison triangle $\cl\triangle(p,x',y')$ in $S^2$ (or $\R^2$, the limits are equal) as $x',y' \to p$ with $x' \in (p,x]$ and $y' \in (p,y]$.

Let $X$ be a $\CAT(0)$ space and fix $x_0 \in X$.
The \defn{boundary}%
, written $\bd X$, of $X$ is the set of all geodesic rays based at $x_0$.
The boundary comes naturally equipped with two topologies.
The \defn{cone topology} on $\bd X$ is the compact-open topology; for $X$ proper, this is compact metrizable.
Another topology on $\bd X$ comes from the \defn{angle metric}, defined by $\angle (p,q) = \sup_{x \in X} \angle_x (p,q)$ for all $p,q \in \bd X$.
This metric induces a finer topology than the cone topology.

The angle metric on $\bd X$ is complete and $\CAT(1)$, but distances are bounded above by $\pi$.
The \defn{Tits metric} is the path metric $\dT$ on $\bd X$ coming from the angle metric; the Tits metric is also complete $\CAT(1)$, and we have the formula $\angle (p,q) = \min \set{\pi, \dT(p, q)}$ for all $p,q \in \bd X$.

An important fact is that the Tits metric is lower semicontinuous with respect to the cone topology; that is, if $p_n \to p \in \bd X$ and $q_n \to q \in \bd X$ under the cone topology, then $\dT(p,q) \le \liminf \dT (p_n, q_n)$.

We adopt the following.

\begin{convention}
When discussing the boundary of a $\CAT(0)$ space, topological properties such as \emph{closed} will refer by default to the cone topology, whereas metric properties will refer to the Tits metric.
Occasionally, to emphasize the cone topology, we write $\bdc X$; for the Tits metric, $\bdT X$; otherwise simply $\bd X$.
\end{convention}

\subsection{Dimension}

Let $Y$ be a $\CAT(1)$ space.
For each $p \in Y$, the \defn{link} $\Lk (p)$ of $p$ in $Y$ is the completion of the set of geodesic germs from $p$ in $Y$, endowed with the angle metric $\angle_p$.
The space $\Lk (p)$ is also $\CAT(1)$.

Following Kleiner \cite{kleiner}, the \defn{geometric dimension} $\dim Y$ of a $\CAT(1)$ space $Y$ is zero if $Y$ is discrete, one if the link of every point in $Y$ is discrete (i.e. geometric dimension zero), two if the link of every point in $Y$ is geometric dimension at most one, and so forth, with $\dim(Y) = \infty$ for the remaining spaces $Y$.
Our main focus will be on $\CAT(1)$ spaces of geometric dimension one.

\begin{convention}
For a $\CAT(1)$ space $Y$, we will write \defn{dimension} to mean the geometric dimension of $Y$, also denoted $\dim Y$.
For a CAT($0$) space $X$, we will write $\dim(\bdT X)$ to mean the geometric dimension of the Tits boundary of $X$.
\end{convention}

If $Y$ is a $\CAT(1)$ space with $\dim Y = 1$ then every metric ball of radius $< \pi$ is an $\R$-tree.
It follows that every connected $\CAT(1)$ space $Y$ such that $\dim Y = 1$ and $H_1 (Y) = 0$ (i.e. zero first Betti number) is an $\R$-tree.
We also obtain the following.

\begin{lemma}			\label{locally injective paths}
Let $Y$ be a $\CAT(1)$ space with $\dim(Y) = 1$.  Every locally injective path in $Y$ can be reparametrized to be locally geodesic. \end{lemma}

Combining this with \lemref{local geodesics}, we obtain the following result.

\begin{corollary}		\label{short paths}
Let $A$ be any subset of a $\CAT(1)$ space $Y$ with $\dim(Y) = 1$.  Then every path in $A$ of length $\le \pi$ can be straightened to a geodesic arc in $A$. \end{corollary}

Kleiner showed \cite[Theorem C]{kleiner} that if $X$ is a proper $\CAT(0)$ space admitting a cocompact action by isometries, then $\dim(\bdT X) < \infty$.

\begin{definition}
Let $Y$ be a $\CAT(1)$ space.
Call a subset $K \subset Y$ a \defn{sphere} if it is isometric (under the Tits metric) to a standard Euclidean sphere of radius one, endowed with the angle metric.  Call a sphere $K \subset Y$ \defn{round} if $\dim K = \dim Y$. \end{definition}

Kleiner proved \cite[Theorem C]{kleiner} that round spheres exist in the Tits boundary $\bdT X$ of $X$ whenever $X$ is proper $\CAT(0)$ admitting a cocompact action by isometries.
Bennett, Mooney, and Spatzier \cite[Corollary 2.5]{bms} proved the following.

\begin{lemma}			\label{density of round spheres}
Let $\G$ be a group acting geometrically on a proper $\CAT(0)$ space $X$, and let $K \subset \bdT X$ be a round sphere.  Then $\G K$ is dense in $\bdc X$. \end{lemma}

\subsection{Involutive sets}

Let $Y$ be a $\CAT(1)$ space.  For $p \in Y$, let $\A(p)$ be the set of \defn{antipodes} $\setp{q \in Y}{\dT(p, q) \ge \pi}$ of $p$ in $Y$.  For $C \subseteq Y$, let $\A(C) = \bigcup_{p \in C} \A(p)$.

\begin{definition}
A set $C \subseteq \bdT X$ is \defn{involutive} if $\A(C) \subseteq C$. \end{definition}

The following rigidity theorem is due to Lytchak.

\begin{theorem} [Main Theorem of \cite{lytchak05}]		\label{lytchak rigidity}
Let $Y$ be a finite-dimensional, geodesically complete $\CAT(1)$ space.
Suppose $A$ is a proper, closed, involutive subset of $Y$.
If $Y$ is geodesically complete, then it as a spherical building or spherical join. \end{theorem}

By Kleiner \cite{kleiner}, the Tits boundary of $X$ is finite-dimensional because $\Gamma$ acts cocompactly on $X$; hence the above theorem applies to $Y = \bdT X$, where $A \subset \bdT X$ is closed under the Tits metric.
However, it is unknown what conditions to place on $X$ to ensure its Tits boundary $\bdT X$ is geodesically complete.

It is well-known (see \cite[Theorem II.9.24]{bridson}) that if $X$ is a complete, geodesically complete $\CAT(0)$ space, and the Tits boundary of $X$ is a spherical join, then $X$ splits as a product.
On the other hand, if the Tits boundary of $X$ is a spherical building, we have the following rigidity theorem due to Leeb.

\begin{theorem}	[\cite{leeb}]		\label{leeb}
Let $X$ be a geodesically complete, proper $\CAT(0)$ space.
If $\bdT X$ is a non-discrete irreducible spherical building under the Tits metric, then $X$ is either a Euclidean building or a higher rank symmetric space.
\end{theorem}

\section{The Group Action}

\begin{standing hypothesis}
For the rest of the paper, let $\G$ be a group acting geometrically (that is: properly discontinuously, cocompactly, and by isometries) on a proper $\CAT(0)$ space $X$.
\end{standing hypothesis}

\subsection{Limit Operators}

The following construction comes from Guralnik and Swenson \cite{gs}:  Let $G$ be a discrete group acting on a compact Hausdorff space $Z$.  Denote by $\beta G$ the Stone--\VV{C}ech compactification of $G$.
For each $z \in Z$, extend the orbit map $\rho_z \colon g \mapsto g z$ and, for $\w \in \beta G$, define $T^{\w} \colon Z \to Z$ by $T^{\w} z = (\beta \rho_z)(\w)$.
Thus, for fixed $z \in Z$, the map $\w \mapsto T^{\w} z$ is a continuous map of $\beta G$ into $Z$ (although $z \mapsto T^{\w} z$ may not be continuous for fixed $\w \in \beta G$).

The family $\set{T^{\w}}_{\w \in \BG}$ of operators is closed under composition.
The inverse map $g \mapsto g^{-1}$ on $G$ extends to a continuous involution $S \colon \beta G \to \beta G$; however, $T^{\w} T^{S\w}$ usually only equals the identity for $\w \in G$.

Now let $\G$ be a discrete group acting properly discontinuously and by isometries on a proper CAT($0$) space $X$ (\emph{proper} as a metric space).  Since $\cl X = X \cup \bdc X$ is compact Hausdorff, the above construction gives us a family of operators $T^{\w}$ on $\cl X$.  Guralnik and Swenson observe that since every $\g \in \G$ acts by isometries on $\bdT X$, every $T^{\w}$ is in fact $1$-Lipschitz on $\bdT X$ by semicontinuity of $\dT$.

\subsection{Folding}

\begin{definition}
Let $K \subset \bdT X$ be a sphere, and suppose $\w \in \BG$ has the property that $T^{\w} (\bd X) = T^{\w} K$ and $T^{\w} \res{K}$ is an isometry.  We say that $\w$ \defn{folds} $K$ onto $T^{\w} K$ (or, $\w$ is a \defn{$K$-folding}), and we call $T^{\w} K$ a \defn{folded} sphere. \end{definition}

Combining results of Guralnik and Swenson \cite[Lemma 3.25]{gs} and Leeb \cite[Proposition 2.1]{leeb}, we obtain the following.

\begin{corollary}
Let $\G$ be a group acting geometrically on a proper $\CAT(0)$ space $X$.
Let $K \subseteq \bdT X$ be a round sphere.  Then some $\w \in \BG$ folds $K$. \end{corollary}

\subsection{$\pi$-convergence}

\begin{definition}
Let $\w \in \BG \setminus \G$.  The \defn{attracting point} $\w^+ = T^{\w}(x) \in \bd X$ and \defn{repelling point} $\w^- = T^{S \w}(x) \in \bd X$ of $\w$ do not depend on $x \in X$. \end{definition}

Papasoglu and Swenson's $\pi$-convergence theorem \cite[Lemma 19]{ps} has a form in terms of $\BG$:

\begin{theorem} [Theorem 3.3 of \cite{gs}]
Let $\G$ be a group acting properly discontinuously by isometries on a \textnormal{CAT($0$)} space $X$, and let $\w \in \BG \setminus \G$.  Then
\[\angle (p, \w^-) + \angle (T^{\w} p, \w^+) \le \pi\]
for all $p \in \bd X$. \end{theorem}

\begin{corollary} [cf.\ Lemma 3.6 of \cite{gs}]		\label{antipodal pi-convergence}
Let $\w \in \BG \setminus \G$.  If $p \in \A(\w^-)$ then $T^{\w} p = \w^+$. \end{corollary}

\begin{definition}
Points $p,q \in \bd X$ are called \defn{$\G$-dual} if there exists some $\w \in \BG \setminus \G$ such that $\w^+ = p$ and $\w^- = q$.  For $p \in \bd X$, let $\D(p)$ denote the set of points $q \in \bd X$ such that $p,q$ are $\G$-dual. \end{definition}

\begin{corollary} [cf.\ Lemma 1.5 of \cite{bb}] 		\label{antipodal pi-convergence 2}
Let $p, q \in \bdT X$ be antipodal.  Then $\D(q) \s \cl{\G p}$ and $\D(p) \s \cl{\G q}$. \end{corollary}

\subsection{Pulling}

A sequence $(x_n)$ in $X$ is said to \defn{radially converge} to $p \in \bd X$ if $x_n \to p$ and, for some (any) fixed geodesic ray $[y, p)$ in $X$, there exists $R > 0$ such that every $x_n$ is distance $\le R$ to the image of $[y, p)$ in $X$.

\begin{definition}
$\w \in \BG$ is said to \defn{pull from} $p \in \bd X$ if there exists a sequence $(\g_n)$ in $\G \subset \BG$ such that $(\g_n)$ accumulates on $S\w \in \BG$ and $(\g_n x)$ radially converges to $p \in \bd X$ for some (any) $x \in X$.
(In particular, $p = \w^-$.)
\end{definition}

\begin{remark}
Guralnik and Swenson \cite{gs} give an equivalent definition of pulling.
\end{remark}

Note by cocompactness we can pull from each $p \in \bd X$.
Also we have (from \cite{gs}):

\begin{proposition}		\label{pulling}
Let $\w \in \BG$ pull from $p \in \bd X$, and let $K \subset \bdT X$ be a round sphere.
\begin{enumerate}
\item  $\angle(T^{\w} p, T^{\w} q) = \angle(p, q)$ for all $q \in \bd X$, and $T^{\w}$ maps $\bd X$ into the suspension $\susp (T^{\w} p, \w^+) = \setp{n \in \bd X}{\dT(T^{\w} p, n) + \dT(n, \w^+) = \pi}$.
\item  If $p \in K$, then $T^{\w} \res{K}$ is an isometry.
In particular, $T^{\w} K$ is a round sphere, and $T^{\w}$ maps the antipode of $p$ on $K$ to the antipode of $T^{\w} p$ on $T^{\w} K$, which equals $\w^+$.
\end{enumerate}
\end{proposition}

\subsection{Splittings and the Centers Lemma}

The next two results come from \cite{ricks-radius}.
The first is part of a more general splitting theorem \cite[Theorem C]{ricks-radius}.

\begin{theorem}			\label{radius rigidity}
Let $\G$ be a group acting geometrically on a proper $\CAT(0)$ space $X$.
Let $K \subset \bd X$ be a round sphere and $A \subset \bd X$ a nonempty closed, $\G$-invariant set.  If $\bd X$ does not split as a spherical join, then $A \cap K$ cannot lie completely in a proper subsphere of $K$. \end{theorem}

Let $(Y,d)$ be any metric space, and let $y \in Y$ and $A \s Y$.
Let $d(y, A) = \inf_{a \in A} d(y, a)$ denote the \defn{metric distance} from $y$ to $A$, and $\Hd(y, A) = \sup_{a \in A} d(y, a)$ the \defn{Hausdorff distance} from $y$ to $A$.
The \defn{radius} (or \defn{circumradius}) of a set $A \s Y$ is $\radius^Y (A) = \inf_{y \in Y} \Hd(y, A)$.
The set of \defn{centers} (or \defn{circumcenters}) of $A$ in $Y$ is $\Centers^Y (A) = \{y \in Y : \Hd(y, A) = \radius^Y (A)\}$.
Notice that $\radius^Y (A)$ and $\Centers^Y (A)$ depend on the ambient space $Y$.

\begin{lemma} [Centers Lemma]	\label{centers lemma}
Let $\G$ be a group acting geometrically on a proper $\CAT(0)$ space $X$.
Let $A$ be a nonempty closed, $\G$-invariant subset of $\bd X$.  For any folded sphere $K \subset \bd X$, the set $K \cap \Centers^{\bd X} (A)$ is a nonempty subset of $\Centers^{K} (A \cap K)$.  Moreover,
\[\radius^{\bd X} (A) = \radius^{K} (A \cap K) = \pi - \max_{K} \dT(\x, A \cap K) = \pi - \sup_{\bd X} \dT(\x, A),
\]
and the supremum is realized. \end{lemma}

\begin{proof}
\lemref{centers lemma} is stated in \cite[Lemma 21]{ricks-radius} with $\angle$ in place of $\dT$.
Since $\angle = \dT$ on $K$, it remains only to show $\sup_{\bd X} \dT(\x, A) \le \pi$.
So suppose, by way of contradiction, that $\dT(p, A) > \pi$ for some $p \in \bd X$.
Then $\angle (p, A) = \pi$, so by Lemma 21 of \cite{ricks-radius}, $\radius^{K} (A \cap K) = 0$.
Thus $A \cap K$ is a single point;
by \thmref{radius rigidity}, $\bd X$ splits as a join.
Hence $\diam (\bdT X) = \pi$, which contradicts $\dT(p, A) > \pi$.
\end{proof}

\subsection{Minimal sets}

\begin{definition}
Call $M \s \bd X$ \defn{minimal} if it is a minimal nonempty, closed, $\G$-invariant subset of $\bd X$. \end{definition}

Note that if $M$ is a minimal set, then $\BG$ acts transitively on $M$.

\begin{lemma}
Let $M,N$ be minimal sets.  For every $m \in M$, there is some $n \in N$ such that $\dT(m, n) = \dT(M, N) := \inf \setp{\dT(p, q)}{p \in M \text{ and } q \in N}$. \end{lemma}

\begin{proof}
Let $M,N$ be minimal and $m \in M$.
Since $M,N$ are closed, by lower semicontinuity of the Tits distance there exist $m' \in M$ and $n' \in N$ such that $\dT(m, n) = \dT(M, N)$.
Because $\BG$ acts transitively on $M$, there is some $\w \in \BG$ such that $T^{\w} m' = m$.
Let $n = T^{\w} n'$.  By $1$-Lipschitzness of $T^{\w}$, $\dT(m, n) = \dT(M, N)$.
\end{proof}

\begin{lemma} [cf. Corollary 1.6 of \cite{bb}]		\label{antipodal to minimal}
Let $M$ be minimal.
If $m \in M$ and $p \in \A(m)$, then $\D(p) = M$ and $\D(m) = \cl{\G p}$. \end{lemma}

\begin{proof}
Let $m \in M$ and $p \in \A(m)$.  Then $\D(p) \s \cl{\G m} = M$ and $\D(m) \s \cl{\G p}$ by \corref{antipodal pi-convergence 2}.  But then $\D(p) = M$ by minimality of $M$, hence $p$ and $m$ are $\G$-dual, and therefore $\cl{\G p} \s \D(m)$.
\end{proof}

\begin{lemma}			\label{closed involutive pair}
Let $M,N$ be minimal.
If $N \cap \A(M)$ is not empty, then $\A(M) \s N$ and $\A(N) \s M$.
In particular, $M \cup N$ is a closed, involutive subset of $\bd X$. \end{lemma}

\begin{proof}
Suppose $m \in M$ and $n \in N \cap \A(m)$.  If $p \in \A(M)$ then $\cl{\G p} = \D(m) = N$ by \lemref{antipodal to minimal}, hence $p \in N$.
\end{proof}

\begin{remark}
If $M,N$ are proper subsets of $\bd X$, it follows that $M \cup N$ is proper, too, as $\dT (M,N) > 0$ or $M = N$.  However, in order to apply \thmref{lytchak rigidity} we need to know that $\bd X$ is geodesically complete.
\end{remark}

\section{One-Dimensional Boundary}

\begin{lemma}			\label{infinite case}
Suppose $\dim(\bdT X) = 1$.  Let $K \s \bd X$ be a round sphere in $\bd X$ and $M$ be a minimal set.  If $M \cap K$ is infinite then $M = \bd X$. \end{lemma}

\begin{proof}
Let $\w \in \BG$ be a $K$-folding, and let $L = T^{\w} K$.
We claim $M \cap L = L$.
So suppose, by way of contradiction, that $M \cap L \neq L$.
Let $N \s \Centers^{\bd X} (M)$ be minimal.
There are only finitely many $p \in L$ that maximize $\dT(\x, M \cap L)$, hence $\Centers^{\bd X} (M) \cap L$ is finite, and thus $N \cap L$ is finite, too.
Now for each $m \in M \cap K$ there exists $n \in N$ such that $\dT(m, n) = \dT(M, N)$.  Applying the map $T^{\w}$, we find $T^{\w} m \in M \cap L$, $T^{\w} n \in N \cap L$, and $\dT(T^{\w} m, T^{\w} n) = \dT(M, N)$.
Now $T^{\w} \res{K}$ is injective, hence there are infinitely many points $T^{\w} m \in M \cap L$.
Each $n' \in N \cap L$ can have $\dT(T^{\w} m, n') = \dT(M, N)$ for at most two $T^{\w} m \in M \cap L$, thus there are infinitely many points in $N \cap L$.
This contradicts finiteness of $N \cap L$; we therefore conclude that $M \cap L = L$, and thus $M = \bd X$ by \lemref{density of round spheres}.
\end{proof}

\begin{standing hypothesis}
From now on, assume $\dim(\bdT X) = 1$ and $\bd X$ is not minimal.  Let $K \s \bd X$ be a folded round circle and $M \s \bd X$ minimal.
\end{standing hypothesis}

\begin{definition}
Let $A \subset K$ be nonempty.  We say that $A$ is \defn{uniformly spaced around $K$} if $K \setminus A$ is the union of open arcs of the same length.  For $A \subset \bd X$, we say $A$ is \defn{uniformly spaced around $K$} if $A \cap K$ is uniformly spaced. \end{definition}

\begin{remark}
If $A \subset K$ is uniformly spaced with $\ell = \abs{A} > 1$, then one can write $A = \set{a_1, a_2, \dotsc, a_\ell}$ where every distance $\dT(a_i, a_{i+1}) = \frac{2\pi}{\ell}$ (taking indices mod $\ell$). \end{remark}

\begin{lemma}			\label{uniformly spaced}
If $M$ minimizes $\abs{M \cap K}$ among minimal sets $M \s \bd X$, then $M$ is uniformly spaced around $K$. \end{lemma}

\begin{proof}
Let $A = \Centers^{\bd X} M$.  By \lemref{centers lemma}, every point of $A \cap K$ is antipodal to a point on $K$ of maximum distance to $M \cap K$.  Thus $\abs{A \cap K} \le \abs{M \cap K}$, with equality if and only if $M \cap K$ is uniformly spaced.
\end{proof}

The following result is an immediate corollary of \lemref{closed involutive pair}.

\begin{lemma}		\label{even involutives}
If $\abs{M \cap K}$ is even, and $M$ is uniformly spaced around $K$, then $\A(M) \s M$.  In other words, $M$ is a closed, involutive subset of $\bd X$. \end{lemma}

\begin{definition}
A set $A \s \bd X$ with at least two elements is called \defn{incompressible} if $\dT(T^{\w} p, T^{\w} q) = \dT(p, q)$ for every $p,q \in A$ and $\w \in \BG$. \end{definition}

Note that if a pair $\set{p,q}$ is incompressible, then so is the whole geodesic arc $[p,q]$.

\begin{lemma}			\label{even case}
If $M$ minimizes $\abs{M \cap K}$ among minimal sets $M \s \bd X$, and $\abs{M \cap K}$ is even, then $\bd X$ either is a spherical building or splits as a spherical join. \end{lemma}

\begin{proof}
Suppose $\abs{M \cap K}$ is even, and $M$ is uniformly spaced around $K$.  Let $C = \Centers^{\bd X} (M)$.  By our hypotheses on $M$, every $c \in C \cap K$ lies on the midpoint of an open arc of $K \setminus A$.  But $N \s C$ for some minimal set $N$, and now $N \cap K$ must contain every such midpoint by choice of $M$.  Thus $\dT(M,N) = \frac{\pi}{\ell}$ for $\ell = \abs{M \cap K}$.  Hence every pair of adjacent $m \in M \cap K$ and $n \in N \cap K$ has $\dT(m,n) = \dT(M,N)$, and is therefore incompressible.  Thus $K$ is covered by incompressible closed arcs $[m,n]$ of length $\frac{\pi}{\ell}$ with endpoints $m \in M$ and $n \in N$.  Therefore, $\bd X$ is also covered by such arcs by \lemref{density of round spheres} and the definition of incompressible.

Now for an incompressible arc $[m,n]$, the interior $(m,n)$ is open in $\bd X$---i.e. there is no branching off $(m,n)$---because otherwise we could pull from some $p \in \bd X$ on such a branch, and $[m,n]$ would get compressed when mapped into the $1$-dimensional suspension $\susp (T^\w p, q) \subset \bd X$.

We now show that $\bd X$ is geodesically complete.  Let $[m,n]$ be an incompressible arc.  By Balser and Lytchak \cite[Lemma 3.1]{bl-centers}, there is some $p \in K \cap \A(n)$.  Since $M \neq \bd X$, $X$ does not have a rank one axis, and therefore $\dT(m,p) \le \pi$ by Ballmann and Buyalo \cite[Proposition 1.10]{bb}; hence the arc $[m,n]$ must extend past $m$ in $\bd X$.  Similarly, $[m,n]$ must extend past $n$ in $\bd X$.  It follows that every geodesic segment in $\bd X$ is extendible in both directions, proving our claim.

By \corref{even involutives}, we see that $M$ is a proper, Tits-closed, involutive subset of $\bd X$.  Therefore, $\bd X$ is a spherical building or spherical join by \thmref{lytchak rigidity}.
\end{proof}

\begin{lemma}			\label{odd case}
Suppose $\dim(\bdT X) = 1$ and $M \subsetneq \bd X$ is a proper minimal subset of $\bd X$.  Let $K \subset \bd X$ be a folded round sphere, and let $\ell = \abs{M \cap K}$.  Suppose that $M$ is chosen to minimize $\ell$, among all minimal sets $M \subset \bd X$, and suppose $\ell \ge 3$ is odd.  Then there exists some $\W \subseteq \bd X$ such that all the following hold:
\begin{enumerate}
\item			\label{odd:convex}
$\W$ is complete and convex under the Tits metric on $\bd X$.
\item			\label{odd:union}
$\W$ is a (not disjoint) union of round circles.
\item			\label{odd:branching}
Branching in $\W$ is restricted to $2 \ell$ points on each circle, uniformly spaced around the circle; these $2 \ell$ points are precisely the points of $M \cap \A(M)$ on the circle.
\item			\label{odd:trees}
$\bd X$ is the union of $\W$ with a (possibly empty) collection of isometrically embedded $\R$-trees.  Each tree $T$ intersects $\W$ at a single point $m = m(T) \in M$, and $\dT(m,p) \le \frac{\pi}{\ell}$ for all $p \in T$.
\end{enumerate}
Furthermore, if $\diam (\bdT X) = \pi$ then $\bd X$ is a spherical building. \end{lemma}

\begin{proof}
By \lemref{uniformly spaced}, $M$ is uniformly spaced around $K$ with distance $\frac{2\pi}{\ell}$ between closest points of $M \cap K$.  Since $\radius^K (M) = \pi - \frac{\pi}{\ell}$ by \lemref{centers lemma}, we see that $\abs{M \cap L} = \ell$ for any folded round circle $L \subset \bd X$.  Thus, pulling from $p \in K$ we see that
\(\Hd(p, M) = \Hd(p, M \cap K)\)
for all $p \in K$.  In particular, the set $\A(M) \cap K$ consists precisely of those points $p \in K$ such that $\dT(p, M \cap K) = \frac{\pi}{\ell}$.

Call any round circle $L \subset \bd X$ a \defn{$\ell$-click circle} if (a) it is the union of $2\ell$ geodesic arcs of length $\frac{\pi}{\ell}$ and of the form $[p,m]$, where $m \in M$ and $p \in \A(M)$, and (b) $\abs{M \cap L} = \abs{\A(M) \cap L} = \ell$.  We have just shown that every folded round circle $L \subset \bd X$ is $\ell$-click.
There are more examples.

\begin{claim}		\label{A(M) implies M}
Let $L$ be a round circle in $\bd X$.
Suppose $A \s \A(M) \cap L$ is uniformly spaced around $L$ with $\abs{A} = \ell$, and $M \cap L$ is nonempty.
Then $L$ is $\ell$-click. \end{claim}

\begin{proof}
Let $A,L$ be as in the hypothesis of the claim, and let $m \in M \cap L$.
Let $o \in L$ be the midpoint of a geodesic arc $[p,q]$ of length $\frac{2\pi}{\ell}$ with $p,q \in A$.
Let $\mu \in \BG$ pull from $o$, let $\nu \in \BG$ fold $T^{\mu} L$; since $\setp{T^{\w}}{\w \in \BG}$ is closed under composition, we may find $\w \in \BG$ such that $T^\w = T^\nu \circ T^\mu$.
Then $T^{\w} L$ is a folded round circle, so it must be $\ell$-click.
But $m \in M \cap L$, so $T^{\w} m \in M \cap T^{\w} L$, and therefore we must have $T^{\w} o \in M \cap T^{\w} L$ because $T^{\w} L$ is $\ell$-click.
But by construction, $(T^{\omega})^{-1} (T^{\w} o) = \set{o}$.
Hence $o \in M$.
Thus $\A(A) \cap L \subset M$, and in particular $\abs{M \cap L} \ge \ell$.
Yet any folding of $L$ contains only $\ell$ points of $M$, so $\abs{M \cap L} = \ell$ and therefore $M \cap L = \A(A) \cap L$.
Likewise, pulling from any $p \in \A(M) \cap L$ and folding yields a folded round circle $T^{\w} L$ such that $\A(T^{\w} p) \in M \cap T^{\w} L$ in addition to every $T^{\w} m' \in M \cap T^{\w} L$ for $m' \in M \cap L$, hence $\A(M) \cap L = A$.
Thus $L$ is $\ell$-click, as claimed.
\end{proof}

Let $\W \s \bd X$ be the union of all geodesic arcs joining points in $M$.
We will show that $\W$ is a union of $\ell$-click circles.

\begin{claim}		\label{W is an lclick union}
$\W$ is a union of $\ell$-click circles, with branching in $\W$ only occurring at points of $M \cup \A(M)$. \end{claim}

\begin{proof}
First, let $L \subset \bd X$ be any $\ell$-click circle and let $m \in M \setminus L$.  Let $x \in L$ be the closest point on $L$ to $m$; since $\radius^{\bd X} (M) < \pi$, we know $\dT(m,x) < \pi$ by \lemref{centers lemma}.  Now the geodesic arc $[m,x]$ can be extended in two distinct ways on $L$, hence there exist distinct $p_1, p_2 \in L$ such that $\dT (m, p_i) = \pi = \dT(m,x) + \dT(x,p_i)$ for $i=1,2$.  But the only points on $L$ equidistant to two $p_i \in \A(M) \cap L$ lie in $M \cup \A(M)$, since $L$ is $\ell$-click.  Therefore, $x \in M \cup \A(M)$.

Now let $m'$ be the closest point of $M \cap L$ to $p_1$ such that $m' \notin [x,p_1]$.  Then $[m,x] \cup [x,p_1] \cup [p_1,m']$ is a locally geodesic arc of length $\pi + \frac{\pi}{\ell}$.  Since $\Hd(m, M) = \Hd(m, M \cap K) = \pi - \frac{\pi}{\ell}$, there is a (unique) geodesic arc $[m,m']$ of length $\le \pi - \frac{\pi}{\ell}$.  Thus the union $L' = [m,x] \cup [x,p_1] \cup [p_1,m'] \cup [m,m']$ is a circle of circumference $\le 2\pi$.  By the CAT($1$) condition, $L'$ has circumference $2\pi$ and thus $[m,m']$ has length $\pi - \frac{\pi}{\ell}$.  We claim that $L'$ is $\ell$-click.

Let $A = \setp{q \in L'}{\dT(q,p_1) = \frac{2\pi k}{\ell} \text{ for some } k = 0,1,2,\dotsc,\frac{\ell-1}{2}}$.
Now, just as $x \in M \cup \A(M)$, it is clear that the point $y \in L \cap [m,m']$ closest to $m$ must lie in $M \cup \A(M)$.
Thus for every $q \in A$, we may find some $n \in M \cap L$ such that $\dT(n,q) = \pi$.
Hence $A \subseteq \A(M) \cap L'$; by \claimref{A(M) implies M}, $L'$ is $\ell$-click, as asserted.

It follows that every $m \in M$ lies on a $\ell$-click circle.  But then, as $L$ and $m$ were arbitrary, every geodesic arc joining points of $M$ must lie on a $\ell$-click circle also.  Thus $\W$ is a union of $\ell$-click circles.
The branching restriction is also clear now.
\end{proof}

This establishes parts \itemref{odd:union} and \itemref{odd:branching} of the lemma.

We now show part \itemref{odd:convex}, i.e. that $\W$ is complete and convex (in the Tits metric).  Note that it follows readily from the definition of $\W$ that it is Tits-closed in $\bd X$, thus $\W$ is complete.  So we consider convexity.

\begin{claim}		\label{convexity}
$\W$ is convex under the Tits metric on $\bd X$. \end{claim}

\begin{proof}
Let $\dW$ be the path metric on $\W$ induced from the Tits metric on $\bd X$; clearly $\dW \ge \dT$ on $\W$.
By \corref{short paths}, for every pair of distinct points $p,q \in \W$ with $\dW (p,q) < \pi$, the unique $\dT$-geodesic arc $[p,q]$ lies completely in $\W$, and therefore $\dW (p,q) = \dT (p,q)$.
Consequently, every $\dW$-geodesic is locally $\dT$-geodesic.
In particular, if $p,q \in \W$ satisfy $\dW (p,q) = \pi$ then $\dT (p,q) = \pi$.
Moreover, due to the highly restricted branching in $\W$, the path metric on $\W$ is realized by geodesic arcs in $\W$---that is, $(\W, \dW)$ is a geodesic metric space.

Now $\W$ is a union of $\ell$-click circles and therefore every $p \in \W$ lies on a geodesic arc $[m,p'] \subset \W$ of length $\frac{\pi}{\ell}$ for some $m \in M$ and $p' \in \A(M)$; thus $\dW (p, M) \le \frac{\pi}{\ell}$ for all $p \in \W$.
By construction, $\dW = \dT$ on $M \times M$, so $\diam_{\W} (M) = \diam_T (M) = \pi - \frac{\pi}{\ell}$.
Therefore, $\diam_{\W} (\W) \le \pi + \frac{\pi}{\ell}$.

Now let $p,q \in \W$, and let $C$ be a geodesic arc in $\bd X$ from $p$ to $q$.
We show that $C \subset \W$; since $p,q$ were arbitrary, this will prove $\W$ is convex.
So suppose, by way of contradiction, that $C$ is not completely contained in $\W$.
By trimming $C$ if necessary, we may assume that $p,q \in \W$ are chosen such that $C \cap \W = \set{p,q}$.
Note that since $C$ is geodesic (not just locally geodesic), $p \neq q$.

Let $r \in \bd X$ be the midpoint of $C$, and let $m \in M$ be the closest point of $M$ to $r$.
By \lemref{centers lemma}, $\dT(r,m) \le \frac{\pi}{\ell}$.
Thus
\begin{align}
\dT(p,m)
&\le \dT(p,r) + \dT(r,m)
\le \frac{1}{2} \dT(p,q) + \frac{\pi}{\ell}
\le \frac{1}{2} \dW(p,q) + \frac{\pi}{\ell} \label{convexity inequality} \\
&\le \frac{1}{2} \left(\pi + \frac{\pi}{\ell} \right) + \frac{\pi}{\ell}
= \frac{\pi}{2} + \frac{3\pi}{2\ell}
\le \pi \nonumber
\end{align}
because $\ell \ge 3$.
Since $\dT(p,r) < \pi$, there is a unique geodesic arc $[p,r]$ in $\bd X$.

Now, since $p \in \W$ and $m \in M$, we know $\dW(p,m) \le \pi$.
We take two cases, $\dW(p,m) = \pi$ and $\dW(p,m) < \pi$.

So first consider the case $\dW(p,m) = \pi$.
Let $D$ be a geodesic arc in $\W$ from $p$ to $q$, and let $E$ be the path in $\bd X$ obtained by concatenating geodesic arcs $[p,r]$ and $[r,m]$ in $\bd X$.
Since $\dW(p,m) = \pi$, we have $\dT(p,m) = \pi$.
This implies we have equality throughout in \itemref{convexity inequality}, and therefore $E$ is a geodesic arc.
Now suppose $E$ and $D$ are distinct.
Then $D \cup E$ must be a round circle.
Since $D \subset \W$, $D$ is the union of $\ell$ geodesic arcs of the form $[n,o]$ with $n \in M$ and $o \in \A(M)$.
Then $D \cup E$ is $\ell$-click by \claimref{A(M) implies M}.
Hence $D \cup E$ is the union of geodesic arcs between points of $M$, and therefore $D \cup E \subset \W$.
As $E \subset \W$ is clear if $E$ and $D$ are not distinct, we have $E \subset \W$ either way.
Thus $[p,r] \subset \W$.
But $[p,r] \subset C$ and $r \notin \set{p,q}$, which contradicts the fact that $C \cap \W = \set{p,q}$.
Therefore, $\dW(p,m) \neq \pi$.

Now suppose $\dW(p,m) < \pi$.
Then the unique $\dT$-geodesic arc $[p,m]$ lies completely in $\W$, and $\dW(p,m) = \dT(p,m)$.
Let $E$ be the path in $\bd X$ obtained by concatenating geodesic arcs $[p,r]$ and $[r,m]$ in $\bd X$.
Since we know $\dT(p,r) + \dT(r,m) \le \pi$ by \itemref{convexity inequality}, straightening $E$ we find $[p,m] \s [p,r] \cup [r,m]$ by \corref{short paths}.
Thus $[p,m] \s C \cup [r,m]$.
Since $C \cap \W = \set{p,q}$ but $[p,m] \subset \W$, it follows that $[p,m] \s [r,m] \cup \set{p,q}$.
So $[p,m] \s [r,m]$.
This implies $\dW(p,m) = \dT(p,m) \le \dT(r,m) \le \frac{\pi}{\ell}$.
By symmetry, we find $\dW(q,m) \le \frac{\pi}{\ell}$ also.
Thus $\dW(p,q) \le \dW(p,m) + \dW(m,q) \le \frac{2\pi}{\ell} < \pi$, hence the unique $\dT$-geodesic arc $[p,q]$ lies completely in $\W$.
This implies $C \subset \W$ and concludes our proof of \claimref{convexity}.
\end{proof}

This proves part \itemref{odd:convex} of the lemma.
We now show part \itemref{odd:trees}.

Let $p \in \bd X \setminus \W$.  By \lemref{centers lemma}, $\dT(p, M) \le \frac{\pi}{\ell}$ and thus $\dT(p, \W) \le \frac{\pi}{\ell}$.
Let $x \in \W$ be the first point in $\W$ on a geodesic arc from $p$ to $M$ of length $\le \frac{\pi}{\ell}$, and let $L \subset \W$ be a $\ell$-click circle containing $x$.
Let $\w \in \BG$ pull from $p$.
Then $T^{\w} \colon \bd X \to \bd X$ is $1$-Lipschitz and maps $\bd X$ into the suspension $\susp (T^{\w} p, \w^+) \subset \bd X$ in such a way that $\dT (T^{\w} q, T^{\w} p) = \angle (q, p)$ for all $q \in \bd X$.
It follows that $T^{\w}$ maps $L$ onto a single geodesic arc of length $\pi$, preserving distances from $X$ (up to distance $\pi - \dT(p,x)$, which is at least $\pi - \frac{\pi}{\ell}$).
Now unless $x \in M \cup \A(M)$, the two closest points $m, m' \in M \cap L$ to $x$ have different distances $a,b \in (0, \frac{2\pi}{\ell})$ to $x$, so $0 < \dT (T^{\w} m, T^{\w} m') < \frac{2\pi}{\ell}$.
Yet we know by \claimref{W is an lclick union} that $\dT(n,n') \ge \frac{2\pi}{\ell}$ for all distinct $n,n' \in M$; therefore, $x \in M \cup \A(M)$.

Now consider $m \in M$ such that $\dT(p,m) \le \frac{\pi}{\ell}$.
Concatenating the geodesic arcs $[x,p]$ and $[p,m]$ gives us a path of length $\le \frac{2\pi}{\ell} < \pi$, so by \corref{short paths} the unique geodesic $[x,m]$ lies completely in $[x,p] \cup [p,m]$.
But $[x,m] \subset \W$ by convexity of $\W$, and $[x,p] \cap \W = \set{x}$ by choice of $x$, so $[p,m]$ must pass through $x$.
Thus $\dT(p,m) = \dT(p,x) + \dT(x,m) > \dT(x,m)$.
Since $\dT(q,M) = \frac{\pi}{\ell}$ for all $q \in \A(M) \cap \W$ and $\dT(p,m) \le \frac{\pi}{\ell}$, we cannot have $x \in \A(M)$.
Thus $x \in M$.

Now let $[m,q]$ be any geodesic arc with $m \in M$ and $[m,q] \cap \W = \set{m}$.
If $\frac{\pi}{\ell} < \dT(m,q) < \frac{2\pi}{\ell}$, there is some $m' \in M \set{m}$ such that $\dT(q,m') \le \frac{\pi}{\ell}$.
Then the concatenated path $[m,q] \cup [q,m']$ has length $< \pi$, hence it contains the unique geodesic $[m,m'] \subset \W$.
But then $[m,q] \cap \W = \set{m}$ implies $[m,m'] \s [q,m']$, which is impossible because $\dT(m,m') \ge \frac{2\pi}{\ell} > \frac{\pi}{\ell} \ge \dT(q,m')$.
Since any geodesic arc $[m,q]$ of length $\ge \frac{2\pi}{\ell}$ must contain points $q'$ with $\frac{\pi}{\ell} < \dT(m,q') < \frac{2\pi}{\ell}$, we conclude that every $q \in \bd X$ with $[m,q] \cap \W = \set{m}$ must have $\dT(q,m) \le \frac{\pi}{\ell}$.

This proves part \itemref{odd:trees} of the lemma.

Finally, suppose $\diam (\bdT X) = \pi$.  Then there can be no $\R$-trees hanging off $\W$; thus $\bd X = \W$ and therefore $\bd X$ is a spherical building \cite[Theorem 6.1]{charney-lytchak}.
\end{proof}

\begin{theorem}[\thmref{MT}]	\label{main theorem}
Suppose $\dim(\bdT X) = 1$ and $M \subsetneq \bd X$ is a proper minimal subset of $\bd X$.  Let $K \subset \bd X$ be a folded round sphere, and let $\ell = \abs{M \cap K}$.  Then $1 \le \ell < \infty$.  Suppose that $M$ is chosen to minimize $\ell$, among all minimal sets $M \subset \bd X$.  If $\ell = 1$, $\bd X$ splits as a suspension.  If $\ell = 2$, $\bd X$ splits as a spherical join.  If $\ell \ge 4$ is even, or if $\ell \ge 3$ is odd and $\diam (\bdT X) = \pi$, then $\bd X$ either is a spherical building or splits as a spherical join. \end{theorem}

\begin{proof}
Because $M$ is nonempty and $K$ is folded, $\ell \ge 1$.
The fact that $\ell < \infty$ follows from \lemref{infinite case}.
The cases $\ell = 1$ and $\ell = 2$ follow from \thmref{radius rigidity}.
The case $\ell \ge 4$, $\ell$ even, is done in \lemref{even case}.
The case $\ell \ge 3$, $\ell$ odd, is done in \lemref{odd case}.
\end{proof}

The number $\ell = \ell(\Gamma) = \min \setp{\abs{M \cap K}}{M \text{ minimal}}$ in \thmref{main theorem} is determined by $\bd X$ as follows.
If $\bd X$ is a spherical join, there is a subgroup $\Gamma_0$ of $\Gamma$ of index at most $2$ that does not permute the join factors.  Thus there is a closed $\Gamma_0$-invariant set $A \subset \bd X$ such that $\abs{A \cap K} \le 2$.
If $\bd X$ is a suspension, there is a subgroup $\Gamma_0$ of $\Gamma$ of index at most $2$ that does not permute the suspension points.  Thus (assuming $\bd X$ is not a circle) there is a closed $\Gamma_0$-invariant set $A \subset \bd X$ such that $\abs{A \cap K} \le 1$.
Thus we have proved the following result.

\begin{theorem}[\thmref{MC3}]	\label{table}
Let $\Gamma$ be a group acting geometrically on a proper $\CAT(0)$ space $X$ with $\dim(\bdT X) = 1$.
Let $K \subset \bd X$ be a folded round sphere.
For each finite-index subgroup $\Gamma_0$ of $\Gamma$, let
$\ell(\Gamma_0) = \inf \abs{M \cap K}$,
where the infimum is taken over minimal nonempty, closed, $\Gamma_0$-invariant subsets $M$ of $\bd X$.

The following table summarizes the complete situation (here $\min \ell(\Gamma_0)$ is taken over all finite-index subgroups $\Gamma_0$ of $\Gamma$):
\textnormal{
\begin{center}
\begin{tabular}{l l l}
\hline
$\bd X$ & possible $\ell(\Gamma)$ & $\min \ell(\Gamma_0)$ \\
\hline
circle & $1$, $2$, $3$, $4$, or $6$ & $1$ \\
suspension but not a circle & $1$ or $2$ & $1$ \\
spherical join but not a suspension & $2$ or $4$ & $2$ \\
irreducible spherical building & integers $\ge 3$ & $\ge 3$ \\
minimal & $\infty$ & $\infty$ \\
$\pi < \diam (\bdT X) \le \pi + \frac{\pi}{\ell(\Gamma)}$ & odd integers $\ge 3$ & odd $\ge 3$ \\
\hline
\end{tabular}
\end{center}} \end{theorem}

\begin{remark}
No example of the last case ($\pi < \diam (\bdT X) \le \pi + \frac{\pi}{\ell(\Gamma)}$) is known, and (at least if $X$ is geodesically complete) would provide a counterexample to the Diameter Rigidity Conjecture stated in \cite{bb}.
\end{remark}

\begin{remark}
The bound $\diam (\bdT X) \le \pi + \frac{\pi}{3}$ coincides with Guralnik and Swenson's bound \cite{gs} for $\dim (\bdT X) = 1$.
\end{remark}

\thmref{MC1} is an immediate corollary of \thmref{main theorem}.

\begin{theorem}[\thmref{MC1}]	\label{dimension 1, diameter pi, boundary statement}
Suppose $\dim(\bdT X) = 1$.  If $\diam (\bdT X) = \pi$ and $\Gamma$ acts minimally on $\bd X$, then $\bd X$ is a spherical building or a spherical join. \end{theorem}

Applying \thmref{leeb}, we obtain the following.

\begin{corollary}[\corref{MC2}]	\label{dimension 1, diameter pi, geodesically complete}
Suppose $\dim(\bdT X) = 1$ and $X$ is geodesically complete.  If $\diam (\bdT X) = \pi$ and $\Gamma$ acts minimally on $\bd X$, then $X$ is a Euclidean building, higher rank symmetric space, or a nontrivial product. \end{corollary}

\bibliographystyle{amsplain}
\bibliography{refs}

\end{document}

%% file: artcfg.tex
\usepackage{amssymb}
\usepackage{hyperref}
\usepackage{tikz}

\usepackage{environ}

\usepackage{import}

\theoremstyle{plain}
\newtheorem{thm}{Theorem}
\newtheorem{theorem}[thm]{Theorem}
\newtheorem*{theorem*}{Theorem}

\newtheorem{corollary}[thm]{Corollary}
\newtheorem{lemma}[thm]{Lemma}
\newtheorem*{lemma*}{Lemma}

\newtheorem*{conjecture*}{Conjecture}

\newtheorem{proposition}[thm]{Proposition}
\newtheorem{claim}{Claim}
\newtheorem*{fact*}{Fact}
\newtheorem*{question*}{Question}

\newtheorem{main}{Theorem}

\newtheorem{maintheorem}[main]{Theorem}

\newtheorem{maincorollary}[main]{Corollary}

\theoremstyle{remark}

\newtheorem*{remark}{Remark}

\theoremstyle{definition}
\newtheorem{definition}[thm]{Definition}
\newtheorem*{convention}{Convention}

\newtheorem*{standing hypothesis}{Standing Hypothesis}

\newcommand{\thmref}[1]{Theorem~\ref{#1}}

\newcommand{\corref}[1]{Corollary~\ref{#1}}

\newcommand{\lemref}[1]{Lemma~\ref{#1}}

\newcommand{\itemref}[1]{\textnormal{(\ref{#1})}}
\newcommand{\itemrefstar}[1]{(\ref*{#1})}

\newcommand{\claimref}[1]{Claim~\ref{#1}}

\newcommand{\defn}[1]{\emph{#1}}

\graphicspath{{pictures/}}

\newcommand{\R}{\mathbb{R}}

\renewcommand{\setminus}{\smallsetminus}

\DeclareMathOperator{\radius}{radius}
\DeclareMathOperator{\diam}{diam}

\DeclareMathOperator{\susp}{\Sigma}

\newcommand{\set}[1]{\left\{#1\right\}}
\newcommand{\setp}[2]{\left\{#1 \mid #2\right\}}
\newcommand{\abs}[1]{\left| #1 \right|}

\newcommand{\res}[1]{\vert_{#1}}

\DeclareMathOperator{\Lk}{Lk}

\newcommand{\cl}[1]{\overline{#1}}
\newcommand{\bd}{\partial}

\DeclareMathOperator{\Dualset}{\mathcal D}
\newcommand{\D}{\Dualset}

\newcommand{\W}{\mathcal W}
\DeclareMathOperator{\A}{\mathcal A}

\DeclareMathOperator{\Centers}{Centers}

\newcommand{\w}{\omega}
\newcommand{\VV}[1]{\ensuremath{\check{\mathrm{#1}}}}
\newcommand{\G}{\Gamma}
\newcommand{\BG}{\beta \Gamma}
\newcommand{\Hd}{\mathcal{H}d}
\newcommand{\g}{\gamma}

\newcommand{\s}{\subseteq}

\newcommand{\x}{-}

\DeclareMathOperator{\dT}{d_T}
\renewcommand{\epsilon}{\varepsilon}

%% file: ricks-one-dim.bbl
\providecommand{\bysame}{\leavevmode\hbox to3em{\hrulefill}\thinspace}
\providecommand{\MR}{\relax\ifhmode\unskip\space\fi MR }
\providecommand{\MRhref}[2]{%
  \href{http://www.ams.org/mathscinet-getitem?mr=#1}{#2}
}
\providecommand{\href}[2]{#2}
\begin{thebibliography}{10}

\bibitem{ballmann}
Werner Ballmann, \emph{Lectures on spaces of nonpositive curvature}, DMV
  Seminar, vol.~25, Birkh\"auser Verlag, Basel, 1995, With an appendix by Misha
  Brin. \MR{1377265 (97a:53053)}

\bibitem{bbrin-2dim}
Werner Ballmann and Michael Brin, \emph{Orbihedra of nonpositive curvature},
  Inst. Hautes \'Etudes Sci. Publ. Math. (1995), no.~82, 169--209 (1996),
  \url{http://www.numdam.org/item?id=PMIHES_1995__82__169_0}. \MR{1383216}

\bibitem{bbrin-3dim}
\bysame, \emph{Rank rigidity of {E}uclidean polyhedra}, Amer. J. Math.
  \textbf{122} (2000), no.~5, 873--885,
  \url{http://muse.jhu.edu/journals/american_journal_of_mathematics/v122/122.5ballmann.pdf}.
  \MR{1781923}

\bibitem{bb}
Werner Ballmann and Sergei Buyalo, \emph{Periodic rank one geodesics in
  {H}adamard spaces}, Geometric and probabilistic structures in dynamics,
  Contemp. Math., vol. 469, Amer. Math. Soc., Providence, RI, 2008,
  \url{http://dx.doi.org/10.1090/conm/469/09159}, pp.~19--27. \MR{2478464
  (2010c:53062)}

\bibitem{bl-centers}
Andreas Balser and Alexander Lytchak, \emph{Centers of convex subsets of
  buildings}, Ann. Global Anal. Geom. \textbf{28} (2005), no.~2, 201--209,
  \url{http://dx.doi.org/10.1007/s10455-005-7277-4}. \MR{2180749}

\bibitem{bms}
Hanna Bennett, Christopher Mooney, and Ralf Spatzier, \emph{Affine maps between
  {$\rm CAT(0)$} spaces}, Geom. Dedicata \textbf{180} (2016), 1--16,
  \url{http://dx.doi.org/10.1007/s10711-015-0087-3}. \MR{3451453}

\bibitem{bridson}
Martin~R. Bridson and Andr{\'e} Haefliger, \emph{Metric spaces of non-positive
  curvature}, Grundlehren der Mathematischen Wissenschaften [Fundamental
  Principles of Mathematical Sciences], vol. 319, Springer-Verlag, Berlin,
  1999. \MR{1744486 (2000k:53038)}

\bibitem{burns-spatzier}
Keith Burns and Ralf Spatzier, \emph{Manifolds of nonpositive curvature and
  their buildings}, Inst. Hautes \'Etudes Sci. Publ. Math. (1987), no.~65,
  35--59, \url{http://www.numdam.org/item?id=PMIHES_1987__65__35_0}. \MR{908215
  (88g:53050)}

\bibitem{caprace-sageev}
Pierre-Emmanuel Caprace and Michah Sageev, \emph{Rank rigidity for {CAT}(0)
  cube complexes}, Geom. Funct. Anal. \textbf{21} (2011), no.~4, 851--891,
  \url{http://dx.doi.org/10.1007/s00039-011-0126-7}. \MR{2827012 (2012i:20049)}

\bibitem{charney-lytchak}
Ruth Charney and Alexander Lytchak, \emph{Metric characterizations of spherical
  and {E}uclidean buildings}, Geom. Topol. \textbf{5} (2001), 521--550,
  \url{http://dx.doi.org/10.2140/gt.2001.5.521}. \MR{1833752}

\bibitem{gs}
Dan~P. Guralnik and Eric~L. Swenson, \emph{A `transversal' for minimal
  invariant sets in the boundary of a {CAT}(0) group}, Trans. Amer. Math. Soc.
  \textbf{365} (2013), no.~6, 3069--3095,
  \url{http://dx.doi.org/10.1090/S0002-9947-2012-05714-X}. \MR{3034459}

\bibitem{kleiner}
Bruce Kleiner, \emph{The local structure of length spaces with curvature
  bounded above}, Math. Z. \textbf{231} (1999), no.~3, 409--456,
  \url{http://dx.doi.org/10.1007/PL00004738}. \MR{1704987}

\bibitem{leeb}
Bernhard Leeb, \emph{A characterization of irreducible symmetric spaces and
  {E}uclidean buildings of higher rank by their asymptotic geometry}, Bonner
  Mathematische Schriften [Bonn Mathematical Publications], 326, Universit\"at
  Bonn, Mathematisches Institut, Bonn, 2000,
  \url{https://arxiv.org/abs/0903.0584}. \MR{1934160}

\bibitem{lytchak05}
A.~Lytchak, \emph{Rigidity of spherical buildings and joins}, Geom. Funct.
  Anal. \textbf{15} (2005), no.~3, 720--752,
  \url{http://dx.doi.org/10.1007/s00039-005-0519-6}. \MR{2221148 (2007d:53066)}

\bibitem{ps}
Panos Papasoglu and Eric Swenson, \emph{Boundaries and {JSJ} decompositions of
  {CAT}(0)-groups}, Geom. Funct. Anal. \textbf{19} (2009), no.~2, 559--590,
  \url{http://dx.doi.org/10.1007/s00039-009-0012-8}. \MR{2545250 (2010i:20046)}

\bibitem{ricks-mixing}
Russell Ricks, \emph{Flat strips, {B}owen-{M}argulis measures, and mixing of
  the geodesic flow for rank one {${\rm CAT}(0)$} spaces}, Ergodic Theory
  Dynam. Systems \textbf{37} (2017), no.~3, 939--970,
  \url{http://dx.doi.org/10.1017/etds.2015.78}. \MR{3628926}

\bibitem{ricks-radius}
\bysame, \emph{Detecting product splittings of {${\rm CAT}(0)$} spaces},
  (2018), preprint, \url{https://arxiv.org/abs/1804.06374}.

\end{thebibliography}
